\documentclass[12pt]{amsart}
\usepackage{amsmath, amssymb, amsthm, mathrsfs, mathtools, amscd}
\usepackage[latin1]{inputenc}
\usepackage{hyperref}
\usepackage[all]{xy}

\newtheorem{theorem}{Theorem}[section]
\newtheorem{lemma}[theorem]{Lemma}
\newtheorem{corollary}[theorem]{Corollary}
\newtheorem{proposition}[theorem]{Proposition}
\newtheorem{definition}[theorem]{Definition}
\newtheorem{remark}[theorem]{Remark}

\newcommand\on[1]{\operatorname{#1}}
\newcommand\mc[1]{\mathcal{#1}}
\newcommand\ps[1]{\underline{#1}}

\newcommand\ld{\lambda}

\newcommand{\sa}{\on{sa}}
\newcommand{\op}{\on{op}}              

\newcommand{\hA}{\hat{A}}
\newcommand\hB{\hat{B}}
\newcommand{\hU}{\hat{U}}
\newcommand{\hP}{\hat{P}}
\newcommand\hQ{\hat{Q}}

\newcommand\PV{\mc{P}(V)}

\newcommand\eq[1]{(\ref{#1})}
\newcommand\Ain[1]{A\,\varepsilon\,#1}

\newcommand{\Sig}{\ps{\Sigma}}            
\newcommand\Subcl[1]{{\rm Sub}_{{\rm cl}}#1} 

\newcommand\Set{\mathbf{Set}}                    
\newcommand\cN{\mc{N}}
\newcommand\cM{\mc{M}}

\newcommand\VN{\mc{V}(\cN)}
\newcommand\SetC[1]{\Set^{#1^{\op}}}
\newcommand\SetVNop{\SetC{\VN}}		
			
\newcommand\SetCAop{\SetC{\CA}}

\newcommand\CA{\mc{C}(\cA)}

\newcommand\bbC{\mathbb{C}}
\newcommand\bbR{\mathbb{R}}
\newcommand\PN{\mathcal{P}(\cN)}

\newcommand\cA{\mc{A}}
\newcommand\cB{\mc{B}}
\newcommand\Ob[1]{\on{Ob}(#1)}

\newcommand\ra{\rightarrow}
\newcommand\mt{\mapsto}
\newcommand\lra{\longrightarrow}
\newcommand\lmt{\longmapsto}

\newcommand\bmeet{\bigwedge}
\newcommand\tphi{\tilde\phi}

\newcommand\SigA{\Sig^\cA}
\newcommand\SigB{\Sig^\cB}
\newcommand\SigN{\Sig^{\cN}}
\newcommand\cG{\mc G}
\newcommand\UA{\mc{U}(\cA)}
\newcommand\phiU{\phi_{\hU}}
\newcommand\PhiU{\Phi_{\hU}}
\newcommand\Id{\on{Id}}
\newcommand\Aut{\on{Aut}}
\newcommand\ol[1]{\overline{#1}}

\newcommand\ga{\gamma}
\newcommand\Ga{\Gamma}
\newcommand\pair[2]{\langle #1,#2\rangle}

\newcommand\uC{\mathbf{uC^*}}

\newcommand\tT{\tilde T}

\newcommand\deo{\delta^o}
\newcommand\Cl{\mc{C}l}
\newcommand\CP{\ps{CP}}
\newcommand\join{\vee}
\newcommand\meet{\wedge}
\newcommand\Rlr{\ps{\bbR^{\leftrightarrow}}}
\newcommand\UN{\mc U(\cN)}
\newcommand\hH{\hat H}
\newcommand\trho{\tilde\rho}
\newcommand\tr{\on{tr}}
\newcommand\io{\iota}

\newcommand\pio{\ps{\io}}
\newcommand\pphi{\ps{\phi}}
\newcommand\ptio{\ps{\tilde\io}}
\newcommand\ptphi{\ps{\tilde\phi}}
\newcommand\de{\delta}

\begin{document}

\title[Flows on Generalised Gelfand Spectra]{Flows on Generalised Gelfand Spectra of Nonabelian Unital $C^*$-Algebras and Time Evolution of Quantum Systems}

\author{Andreas D\"oring}

\address{Andreas D\"oring, Clarendon Laboratory, Department of Physics, University of Oxford, Parks Road, OX1 3PU, Oxford, UK}
\email{doering@atm.ox.ac.uk}
\date{4. December 2013}

\begin{abstract}
In \cite{Doe12b}, we associated a presheaf $\SigA$ with each unital $C^*$-algebra $\cA$. The spectral presheaf $\SigA$ generalises the Gelfand spectrum of an abelian unital $C^*$-algebra. In the present article, we consider one-parameter groups of automorphisms of the spectral presheaf, in particular those arising from one-parameter groups of inner automorphisms of the algebra. We interpret the spectral presheaf as a (generalised) state space for a quantum system and show how we can use flows on the spectral presheaf and on associated structures to describe the time evolution of non-relativistic quantum systems, both in the Schr\"odinger picture and the Heisenberg picture. 
\end{abstract}

\maketitle




\section{Introduction}
It was shown in \cite{Doe12b} how to associate a presheaf $\SigA$, called the spectral presheaf, with each unital $C^*$-algebra $\cA$. This object was first defined in the topos approach to quantum theory \cite{IshBut98,IHB00,DoeIsh08a,DoeIsh08b,DoeIsh08c,DoeIsh08d,Doe09a,Doe09b,DoeIsh11,Doe11a,Doe11b,DoeIsh12,DoeBar12,Doe12}, see also \cite{HLS09a,HLS09b,HLS11,Wol10}, and is interpreted physically as a generalised state space of a quantum system. Mathematically, the spectral presheaf can be regarded as a generalised Gelfand spectrum of a nonabelian $C^*$-algebra.

It was shown in \cite{Doe12b} that every unital $*$-homomorphism between unital $C^*$-algebras gives rise to a morphism between their presheaves in the opposite direction, and how the construction of spectral presheaves and their morphisms can be understood categorically as being based on `local duality'. By considering automorphisms of the spectral presheaf, we determined how much algebraic information about a nonabelian $C^*$-algebra is contained in its spectral presheaf. For a von Neumann algebra $\cN$ not isomorphic to $\bbC^2$ and with no type $I_2$ summand, we showed that the spectral presheaf $\SigN$ determines exactly the Jordan $*$-structure of $\cN$, and using recent results by Hamhalter, we showed that for a unital $C^*$-algebra $\cA$, the spectral presheaf determines $\cA$ up to quasi-Jordan isomorphisms and, for a large class of $C^*$-algebras, also up to Jordan isomorphims. 

In the present article, we will develop some further structures relating to the spectral presheaf, viz. flows corresponding to one-parameter groups of automorphisms, and their applications in physics. On the physical side, the setting is non-relativistic algebraic quantum theory in which the physical quantities (observables) of a single quantum system are described by the self-adjoint operators in a von Neumann algebra. This is what we mean when referring to `standard quantum theory' in the main text. Extensions to composite systems, special relativistic space-times, etc. will be treated in future work.

In order to make this article reasonably self-contained, we will state those results of \cite{Doe12b} that we will need in the following in section \ref{Sec_SomeResultsOfPartI}. In section \ref{Sec_UnitariesAndFlows}, we briefly consider inner automorphisms of a unital $C^*$-algebra, one-parameter groups of unitaries and associated flows on the spectral presheaf. In section \ref{Sec_TimeEvolution}, which forms the bulk of the paper, we discuss the time evolution of quantum systems in terms of flows on structures associated with the spectral presheaf of (the von Neumann algebra of physical quantities of) a quantum system. Prop. \ref{Prop_IsomS(N)AndGaCP} may be of independent interest, since it is a reformulation of the generalised version of Gleason's theorem, valid for von Neumann algebras with no type $I_2$ summand, in terms of global sections of a certain presheaf. This can be compared to a similar result by de Groote \cite{deG07} and the reformulation of the Kochen-Specker theorem in terms of (the non-existence of) global sections of the spectral presheaf \cite{IshBut98,IHB00,Doe05}, which was a central insight in the early stages of the topos approach to quantum theory.

In section \ref{Sec_Compatibility}, we show that our reformulations of the the Heisenberg and the Schr\"odinger picture in terms of flows on the spectral presheaf are compatible, in analogy to standard quantum theory.

\section{Some previous results}			\label{Sec_SomeResultsOfPartI}
We summarise notation and sketch some results from \cite{Doe12b} that we will need in the rest of this paper. For details, proofs and a fuller development, see there. Standard references on operator algebras are e.g. \cite{KadRin83/86,Bla06}, and \cite{McLMoe92,Joh02/03} on topos theory.


Let $\cA$ be a unital $C^*$-algebra. The set $\CA$ of unital abelian $C^*$-subalgebras of $\cA$ that share the unit element with $\cA$, partially ordered under inclusion, is called the \emph{context category of $\cA$}. Elements $C,D$ of $\CA$ are called \emph{contexts of $\cA$}. The category $\SetCAop$ of presheaves over $\CA$, that is, contravariant functors from $\CA$ to $\Set$, with natural transformations as arrows, is a topos. We call it the \emph{(presheaf) topos associated with $\cA$}. The \emph{spectral presheaf $\SigA$} is an object in the topos $\SetCAop$, given
\begin{itemize}
	\item [(a)] on objects: for all contexts $C\in\CA$, $\SigA_C$ is the Gelfand spectrum of $C$, the set of algebra morphisms (characters) $\ld:C\ra\bbC$, equipped with the relative weak$^*$-topology,
	\item [(b)] on arrows: for all inclusions $i_{C'C}:C'\hookrightarrow C$ in $\CA$,
	\begin{align}
				\SigA(i_{C'C}): \SigA_C &\lra \SigA_{C'}\\			\nonumber
				\ld &\lmt \ld|_{C'}.
	\end{align}
\end{itemize}

\begin{definition}			\label{Def_AutomOfSigA}
(Def. 5.16 in \cite{Doe12b}) Let $\cA$ be a unital $C^*$-algebra, and let $\SigA$ be its spectral presheaf. An \emph{automorphism of $\SigA$} is a pair $\pair{\Ga}{\io}$, where $\Ga:\SetCAop\ra\SetCAop$ is an essential geometric automorphism, induced by an order-automorphism $\ga:\CA\ra\CA$ (called the \emph{base map}). $\Ga^*:\SetCAop\ra\SetCAop$ is the inverse image functor of the geometric automorphism $\Ga$, and $\io:\Ga^*(\SigA)\ra\SigA$ is a natural isomorphism for which each component $\io_C:(\Ga^*(\SigA))_C\ra\SigA_C$, $C\in\CA$, is a homeomorphism. Hence, an automorphism $\pair{\Ga}{\io}$ acts by
\begin{equation}
			\SigA \stackrel{\Ga^*}{\lra} \Ga^*(\SigA) \stackrel{\io}{\lra} \SigA.
\end{equation}
We will also use the notation $\io\circ\Ga^*$ for an automorphism $\pair{\Ga}{\io}$. The automorphisms of $\SigA$ form a group, which we denote by $\Aut(\SigA)$.
\end{definition}

The first key observation is that every unital $*$-automorphism $\phi:\CA\ra\CA$ induces an automorphism of $\SigA$: clearly, $\phi$ gives an order-automorphism (base map)
\begin{align}
			\tphi:\CA &\lra\ \CA\\			\nonumber
			C &\lmt \phi(C),
\end{align}
which in turn induces an essential geometric morphism $\Phi:\SetCAop\ra\SetCAop$ with inverse image part
\begin{align}
			\Phi^*:\SetCAop &\lra \SetCAop\\			\nonumber
			\ps P &\lmt \ps P\circ\tphi.
\end{align}
In particular, $\Phi^*(\SigA)$ is given, for each $C\in\CA$, by
\begin{equation}
			(\Phi^*(\SigA))_C=\SigA_{\tphi(C)}=\Sigma(\phi(C)),
\end{equation}
i.e., the component of $\Phi^*(\SigA)$ at $C$ is the Gelfand spectrum of $\phi(C)$. Morevoer, for each $C\in\CA$, there is a unital $*$-isomorphism
\begin{equation}
			\phi|_C: C \lra \phi(C),
\end{equation}
and by Gelfand duality, there is a homeomorphism
\begin{align}
			\cG_{\phi;C}:\Sigma(\phi(C))=\SigA_{\tphi_C} &\lra \SigA_C=\Sigma_C\\			\nonumber
			\ld &\lmt \ld\circ\phi|_C.
\end{align}
The $\cG_{\phi;C}$, $C\in\CA$, are the components of a natural isomorphism
\begin{equation}
			\cG_\phi:\Phi^*(\SigA) \lra \SigA,
\end{equation}
so we obtain an automorphism $\pair{\Phi}{\cG_\phi}:\SigA\ra\SigA$, acting by
\begin{equation}
			\SigA \stackrel{\Phi^*}{\lra} \Phi^*(\SigA) \stackrel{\cG_\phi}{\lra} \SigA.
\end{equation}
It is easy to see that the assignment $\phi\mt\pair{\Phi}{\cG_\phi}$ is injective. We obtain:

\begin{proposition}			\label{Prop_RepOfAutcAInAutSigA}
(Prop. 5.23 in \cite{Doe12b}) Let $\cA$ be a unital $C^*$-algebra, and let $\SigA$ be its spectral presheaf. There is an injective group homomorphism
\begin{align}
			\Aut(\cA) &\lra \Aut(\SigA)^{\op}\\			\nonumber
			\phi &\lmt \pair{\Phi}{\cG_\phi}=\cG_{\phi}\circ\Phi^*
\end{align}
from the automorphism group of $\cA$ into the opposite group of automorphism group of $\SigA$. (This is the same as an injective, contravariant group homomorphism from $\Aut(\cN)$ into $\Aut(\SigN)$.)
\end{proposition}

In \cite{Doe12b}, we showed that if the unital $C^*$-algebra $\cA$ is neither isomorphic to $\bbC^2$ nor to $\mc B(\bbC^2)$, then in fact an automorphism $\pair{\Ga}{\io}$ of $\SigA$ is completely determined by specifying the base map $\ga:\CA\ra\CA$ underlying the essential geometric morphism $\Ga$. (The proof uses a result by Hamhalter \cite{Ham11}.) That is, the spectral presheaf $\SigA$ is `rigid' in the sense that it has exactly as many automorphisms as the underlying poset $\CA$. In order to formulate this precisely, let $\Aut_{ord}(\CA)$ denote the group of order automorphisms of the context category. Then we have

\begin{theorem}			\label{Thm_GroupIsosCStar}
(Thm. 5.27, Cor. 5.28 in \cite{Doe12b}) Let $\cA$ be a unital $C^*$-algebra such that $\cA$ is neither isomorphic to $\bbC^2$ nor to $\mc B(\bbC^2)$. The group $\Aut_{ord}(\CA)$ is contravariantly isomorphic to the group $\Aut(\SigA)$ of automorphisms of the spectral presheaf $\SigA$ of $\cA$.
\end{theorem}

Given a von Neumann algebra $\cN$, the set $\VN$ of non-trivial abelian von Neumann subalgebras that share the unit element with $\cN$, partially ordered under inclusion, is called the \emph{context category of the von Neumann algebra $\cN$}. Elements $V,W$ of $\VN$ are called \emph{contexts of $\VN$}.

The spectral presheaf $\SigN$ of a von Neumann algebra $\cN$ is a presheaf over $\VN$, defined in complete analogy to the case of unital $C^*$-algebras (where the base category is $\CA$). Automorphisms of the spectral presheaf $\SigN$ of a von Neumann algebra are defined just as in the case of unital $C^*$-algebras, cf. Def. \ref{Def_AutomOfSigA}. The analogue of Prop. \ref{Prop_RepOfAutcAInAutSigA} also holds:

\begin{proposition}			\label{Prop_RepOfAutcNInAutSigN}
(Prop. 5.5 in \cite{Doe12b}) Let $\cN$ be a von Neumann algebra, and let $\SigN$ be its spectral presheaf. Let $\Aut(\cN)$ be the automorphism group of $\cN$, and let $\Aut(\SigN)$ be the automorphism group of $\SigN$. There is an injective group homomorphism from $\Aut(\cN)$ to $\Aut(\SigN)^{op}$ (that is, there is an injective, contravariant group homomorphism from $\Aut(\cN)$ into $\Aut(\SigN)$), given by
\begin{align}
			\Aut(\cN) &\lra \Aut(\SigN)^{op}\\			\nonumber
			\phi &\lmt \pair{\Phi}{\cG_\phi}=\cG_{\phi}\circ\Phi^*
\end{align}
\end{proposition}

If $\cN$ is a von Neumann algebra with projection lattice $\PN$, then the group of automorphisms of the complete orthomodular lattice $\PN$ is denoted $\Aut_{cOML}(\PN)$. Moreover, each von Neumann algebra $\cN$ determines a Jordan algebra, also denoted $\cN$, which has the same elements and linear structure as the von Neumann algebra, and Jordan product defined by
\begin{equation}
			\forall \hA,\hB\in\cN: \hA\cdot\hB = \frac{1}{2}(\hA\hB+\hB\hA).
\end{equation}
A Jordan automorphism of $\cN$ is a bijective (unital) linear map that preserves the Jordan product. The Jordan automorphisms of $\cN$ form a group $\Aut_{Jordan}(\cN)$. 

In \cite{Doe12b}, it was shown:


\begin{theorem}			\label{Thm_GroupIsos}
(Cor. 5.15 in \cite{Doe12b}) Let $\cN$ be a von Neumann algebra not isomorphic to $\bbC^2$ and without summand of type $I_2$, with projection lattice $\PN$, context category $\VN$, associated partial von Neumann algebra $\cN_{part}$, associated Jordan algebra also denoted $\cN$, and spectral presheaf $\SigN$.

The four groups $\Aut_{ord}(\VN)$, $\Aut_{cOML}(\PN)$, $\Aut_{part}(\cN_{part})$ and $\Aut_{Jordan}(\cN)$ are isomorphic. Concretely, every order-isomorphism (base map) $\tT\in\Aut_{ord}(\VN)$ induces a unique automorphism $T:\PN\ra\PN$ of the complete orthomodular lattice of projections, which extends to a partial von Neumann automorphism $T:\cN_{part}\ra \cN_{part}$ and further to a Jordan $*$-automorphism $T:\cN\ra\cN$. Conversely, each Jordan $*$-automorphism restricts to an automorphism of the partial algebra $\cN_{part}$, further to an automorphism of $\PN$, and induces an order-automorphism of $\VN$.

Each base map $\tT:\VN\ra\VN$ induces an automorphism $\pair{\tT}{T^*}:\SigN\ra\SigN$ of the spectral presheaf, and the group $\Aut(\SigN)$ of automorphisms of the spectral presheaf is contravariantly isomorphic to the groups $\Aut_{ord}(\VN)$, $\Aut_{cOML}(\PN)$, $\Aut_{part}(\cN_{part})$ and $Aut_{Jordan}(\cN)$.
\end{theorem}

\paragraph{Simplifying notation} Let $\cA$ be a unital $C^*$-algebra, and let $\SigA$ be its spectral presheaf. If $\pair{\Ga}{\io}:\SigB\ra\SigA$ is an isomorphism (or $\pair{\Ga}{\io}:\SigA\ra\SigA$ is an automorphism), we will often simply write
\begin{align}
			\pio : \SigA\lra\SigA
\end{align}
instead of $\pair{\Ga}{\io}$, leaving the essential geometric automorphism $\Ga$ implicit. If $\phi:\cA\ra\cA$ is an algebra automorphism, then we will use the notation $\pphi:\SigA\ra\SigA$ for the automorphism induced by it (so $\pphi=\pair{\Phi}{\cG_\phi}:\SigA\ra\SigA$).

Analogous remarks apply for automorphisms of the spectral presheaf of a von Neumann algebra.

\section{Action of the unitary group and flows on the spectral presheaf}			\label{Sec_UnitariesAndFlows}
In subsection \ref{Subsec_InnerAutoms}, we consider some aspects of inner automorphisms and the action of the unitary group $\UA$ of a unital $C^*$-algebra $\CA$ on its spectral presheaf $\SigA$. In subsection \ref{Subsec_FlowsOnSig}, flows on the spectral presheaf are defined. It is shown that each one-parameter group of unitaries induces a flow on $\Sig$.

\subsection{Inner automorphisms}			\label{Subsec_InnerAutoms}
Let $\cA\in\Ob\uC$ be a unital $C^*$-algebra, and let $\UA$ be the unitary group of $\cA$. Each element $\hU\in\UA$ induces an inner automorphism
\begin{align}
			\phiU:\cA &\lra \cA\\			\nonumber
			\hA &\lmt \hU\hA\hU^*
\end{align}
of the algebra $\cA$. This corresponds to an automorphism of $\SigA$, the spectral presheaf of $\cA$. Concretely, $\phiU$ induces an order automorphism
\begin{align}
			\widetilde{\phiU}:\CA &\lra \CA\\			\nonumber
			C &\lmt \hU C\hU^*,
\end{align}
which gives an essential geometric automorphism
\begin{equation}
			\PhiU:\SetCAop \lra \SetCAop.
\end{equation}
Using the inverse image part, we map $\SigA$ to $\PhiU^*(\SigA)$. We then consider the natural isomorphism
\begin{equation}
			\cG_{\hU}:\PhiU^*(\SigA) \lra \SigA
\end{equation}
with components
\begin{align}
			\cG_{\hU;C}:\PhiU^*(\SigA)_C &\lra \SigA_C\\			\nonumber
			\ld &\lmt \ld\circ\phiU|_C.
\end{align}
Then
\begin{equation}
			\pair{\PhiU}{\cG_{\hU}}=\cG_{\hU}\circ\PhiU^*:\SigA \lra \SigA
\end{equation}
is the automorphism of $\SigA$ induced by $\hU\in\UA$. We will often write 
\begin{align}
			\pphi_{\hU}:\SigA \lra \SigA
\end{align}
for this automorphism. It is clear by construction that 
\begin{equation}
			\pphi_{\hU}^{-1}=(\pair{\PhiU}{\cG_{\hU}})^{-1}=\pair{\Phi_{\hU^*}}{\cG_{\hU^*}}=\pphi_{\hU^*}
\end{equation}
is the inverse automorphism, and that $\pphi_{\hat 1}=\pair{\Phi_{\hat 1}}{\cG_{\hat 1}}=\Id_{\SigA}$.

Prop. \ref{Prop_RepOfAutcAInAutSigA} implies
\begin{proposition}			\label{Prop_AntiRepOfUAInAutSigA}
Let $\cA$ be a unital $C^*$-algebra, and let $\SigA$ be its spectral presheaf. There is a group homomorphism
\begin{align}
			\UA &\lra \Aut(\SigA)^{\op}\\			\nonumber
			\hU &\lmt \pphi_{\hU}=\pair{\PhiU}{\cG_{\hU}}
\end{align}
from the unitary group of $\cA$ into the opposite of the group of automorphisms of the spectral presheaf $\SigA$ of $\cA$. (This is the same as a contravariant group homomorphism from $\UA$ to $\Aut(\SigA)$.)
\end{proposition}

This group homomorphism is not injective, though, because if $\hU$ is in the center $\mc Z(\cA)$ of $\cA$, then the inner automorphism $\phi_{\hU}:\cA\ra\cA$ induced by $\hU$ is the identity. It is straightforward to remedy this: let $\UA_0:=\UA\cap\mc Z(\cA)$, and let\footnote{The following quotient is not trivial in physical terms, at least not if more sophisticated situations are considered: in the case of multi-particle systems and degeneracy, the trivial dynamics (described by $\UA_0$) may lead to relative phases. Other interpretational issues arise from non-trivial experiments involving weak measurements. I thank Bertfried Fauser for pointing out these aspects to me. Here, we focus on a single system and its unitary evolution, for which dividing out $\UA_0$ is unproblematic.}
\begin{equation}
			\UA_{\on{proper}}:=\UA/\UA_0.
\end{equation}
We obtain:

\begin{corollary}			\label{Cor_FaithfulAntiRepOfUA}
Let $\cA$ be a unital $C^*$-algebra, and let $\SigA$ be its spectral presheaf. There is an injective group homomorphism
\begin{align}
			\UA_{\on{proper}} &\lra \Aut(\SigA)^{\op}\\			\nonumber
			[\hU] &\lmt \pphi_{\hU}=\pair{\PhiU}{\cG_{\hU}},
\end{align}
where $\hU$ is a representative of the equivalence class $[\hU]$.
\end{corollary}

\subsection{Flows on the spectral presheaf}			\label{Subsec_FlowsOnSig}
We switch to von Neumann algebras from here on. This can be seen as moving from a unital $C^*$-algebra to its enveloping von Neumann algebra. It is clear that each element $\hU$ of the unitary group $\UN$ of a von Neumann algebra $\cN$ induces an inner automorphism $\phi_{\hU}$ of $\cN$, and that the analogues of Prop. \ref{Prop_AntiRepOfUAInAutSigA} and Cor. \ref{Cor_FaithfulAntiRepOfUA} hold for von Neumann algebras.

Let $\hH$ be a self-adjoint operator affiliated with $\cN$, possibly unbounded (for the theory of unbounded self-adjoint operators, see e.g. \cite{KadRin83/86}). By Stone's theorem, $\hH$ induces a strongly continuous one-parameter group
\begin{align}
			U: \bbR &\lra \UN\\			\nonumber
			t &\lmt e^{it\hH}
\end{align}
of unitary operators in $\cN$, i.e., a representation of $\bbR$ in $\UN$. Each $\hU_t:=U(t)$ gives an inner automorphism
\begin{align}
			\phi_{\hU_t}: \cN &\lra \cN\\			\nonumber
			\hA &\lmt \hU_t\hA\hU_{t}^*.
\end{align}
This suggests the following definition:
\begin{definition}
A representation
\begin{align}
			F: \bbR &\lra \Aut(\SigN)\\			\nonumber
			t &\lmt \pio_t = \pair{\Ga_t}{\io_t}
\end{align}
of the additive group of real numbers by automorphisms of the spectral presheaf is called a \emph{flow on the spectral presheaf $\Sig$}.
\end{definition}

Note that if $F$ is a flow on $\SigN$, then $F(0)=\Id_{\SigN}$, the identity automorphism on $\SigN$. Let $\hH$ be a self-adjoint operator affiliated with $\cN$, let $(\hU_t)_{t\in\bbR}$ be the strongly continuous one-parameter group of unitaries in $\cN$ given by $\hU_t=e^{it\hH}$ for all $t$, and let $(\phi_{\hU_t})_{t\in\bbR}$ be the one-parameter group of inner automorphisms of $\cN$ corresponding to $(\hU_t)_{t\in\bbR}$. Then $(\phi_{\hU_t})_{t\in\bbR}$ is a representation of $\bbR$ in $\Aut(\cN)$, and by Prop. \ref{Prop_RepOfAutcNInAutSigN}, there is a group homomorphism from $\Aut(\cN)$ into $\Aut(\SigN)^{\op}$, so $(\phi_{\hU_t})_{t\in\bbR}$ (resp. $(\hU_t)_{t\in\bbR}$) determines a flow
\begin{align}
			F_{\hH}: \bbR &\lra \Aut(\SigN)\\			\nonumber
			t &\lmt \pphi_{\hU_t}=\pair{\Phi_{\hU_t}}{\cG_{\hU_t}}
\end{align}
More generally, by Thm. \ref{Thm_GroupIsos} every one-parameter group $(J_t)_{t\in\bbR}$ of Jordan $*$-automorphisms of $\cN$ determines a flow on the spectral presheaf $\SigN$.

\section{Flows and time evolution of physical systems}			\label{Sec_TimeEvolution}
It is interesting to consider flows defined not only on the spectral presheaf, but also on other structures associated with it, because this allows describing the time evolution of quantum systems. In the picture we will develop in this section, the spectral presheaf is interpreted as a state space of the quantum system, analogous to the state space (phase space) of a classical system. 

As usual in quantum theory, there are two formulations of time evolution: a Schr\"odinger picture, in which states change in time and physical quantities stay fixed, and a Heisenberg picture, in which physical quantities change in time and states remain fixed. We will discuss the analogous transformations with respect to structures associated with the spectral presheaf. We will consider a von Neumann algebra $\cN$, interpreted as the algebra of physical quantities of a quantum system, and its spectral presheaf $\SigN$, interpreted physically as a state space for the quantum system. Tom Woodhouse developed some of these aspects in his M.Sc. thesis \cite{Woo11}.

In fact, instead of considering the representation of physical quantities directly (which in the topos approach to quantum theory are given by certain arrows $\breve\delta(\hA):\SigN\ra\Rlr$ from the spectral presheaf to a presheaf $\Rlr$ of generalised values, permitting real intervals as `unsharp' values, see \cite{DoeIsh08c,Doe11b}), we will be concerned with \emph{propositions} about the values of physical quantities. The basic propositions are written symbolically as ``$\Ain\de$'', which is interpreted as ``the physical quantity $A$ has a value in the Borel set $\io\subseteq\bbR$''. Mathematically, such propositions are represented by suitable subobjects of the spectral presheaf $\SigN$. In the Schr\"odinger picture, the propositions are fixed, but in the Heisenberg picture, they change in time, because the physical quantities do.

States of the von Neumann algebra $\cN$ correspond to (generalised) probability measures on the spectral presheaf $\SigN$. This is structurally analogous to classical physics, where states are given by probability measures on the state space of the system. We discuss the Heisenberg picture in subsection \ref{Subsec_HeisenbergPic}, the Schr\"odinger picture in subsection \ref{Subsec_SchroedingerPic}. Compatibility between them will be discussed in the section \ref{Sec_Compatibility}.

\subsection{Propositions, clopen subobjects and Heisenberg picture}			\label{Subsec_HeisenbergPic}
Let $A$ be a physical quantity of the quantum system under consideration, and let $\hA\in\cN_{\sa}$ be the self-adjoint operator in the von Neumann algebra $\cN$ that represents $A$.\footnote{We assume that all relevant physical quantities can be represented by bounded operators in $\cN$. It is straightforward to treat also unbounded operators affiliated with $\cN$, since we just consider propositions such as ``$\Ain\de$'' about the physical quantities, which by the spectral theorem correspond to projection operators. For a self-adjoint operator $\hA$ affiliated with $\cN$, all these projections lie in the algebra $\cN$.} Let ``$\Ain\de$'' be a proposition about the value of $A$. Since we aim to formulate a form of Heisenberg picture, we know that $A$ and hence ``$\Ain\de$'' will change in time, so we will add a time label $t\in\bbR$ to the proposition, now writing ``$\Ain\de;t$''.\footnote{Alternatively, we could use the notation ``$A(t)\in\io$'', which would suggest more clearly that the physical quantity $A$ is changing in time. In any case, this is just a symbolic notation; the key point of course is how such propositions are represented.}

By the spectral theorem, a proposition ``$\Ain\de;t$'' corresponds to a projection $\hP_t\in\PN$, the projection lattice of $\cN$. For each $V\in\VN$, let
\begin{equation}			\label{Def_OuterDasOfProjs}
			\deo_V(\hP_t):=\bmeet\{\hQ\in\PV \mid \hQ\geq\hP_t\}.
\end{equation}
For each abelian von Neumann algebra $V$, there is an isomorphism
\begin{align}			\label{Def_alphaV}
			\alpha_V: \PV &\lra \Cl(\SigN_V)\\			\nonumber
			\hP &\lmt \{\ld\in\Sigma_V \mid \ld(\hP)=1\}
\end{align}
between the complete Boolean algebra of projections in $V$ and the complete Boolean algebra $\Cl(\SigN_V)$ of clopen subsets of the Gelfand spectrum of $V$. Given a projection $\hP\in\PV$, we will write $S_{\hP}:=\alpha_V(\hP)$ for the corresponding clopen subset of $\SigN_V$. Conversely, given $S\in\Cl(\SigN_V)$, we will write $\hP_S:=\alpha_V^{-1}(S)$ for the corresponding projection in $V$.

In particular, the projection $\deo_V(\hP_t)\in\PV$ corresponds to the clopen subset
\begin{equation}
			S_{\deo_V(\hP_t)}=\alpha_V(\deo_V(\hP_t))\subseteq\SigN_V.
\end{equation}
Letting $V$ vary over $\VN$, we obtain a family $(S_{\deo_V(\hP_t)})_{V\in\VN}$ of clopen subsets, one in each component $\SigN_V$ of the spectral presheaf. It is straightforward to show that these clopen subsets form a subobject of the spectral presheaf $\SigN$, that is, for all inclusions $i_{V'V}:V'\hookrightarrow V$, it holds that
\begin{equation}
			\SigN(i_{V'V})(S_{\deo_V(\hP_t)})=\{\ld|_{V'} \mid \ld\in S_{\deo_V(\hP_t)}\} \subseteq S_{\deo_{V'}(\hP_t)}.
\end{equation}
In fact, equality holds, $\SigN(i_{V'V})(S_{\deo_V(\hP_t)})=S_{\deo_{V'}(\hP_t)}$, as was shown in \cite{DoeIsh08b}. We write
\begin{equation}			\label{Def_OuterDasOfProj}
			\ps\deo(\hP_t):=(S_{\deo_V(\hP_t)})_{V\in\VN}
\end{equation}
for this subobject, which is called the \emph{outer daseinisation of $\hP_t$}. The subobject $\ps\deo(\hP_t)$ is the representative of the proposition ``$\Ain\de;t$'' in the topos approach.\footnote{In classical physics, a proposition like ``$\Ain\de;t$'' is represented by a (Borel) subset $S_t$ of the state space $\Sigma$ of the system. The spectral presheaf $\SigN$ is the analogue of the state space $\Sigma$ for the quantum case, and the subobject $\ps S_t=\ps\deo(\hP_t)$ is the analogue of the subset $S_t$.} For a detailed discussion of these aspects, see \cite{DoeIsh08b,Doe11b,Doe12}.

\begin{definition}
A subobject $\ps S$ of $\SigN$ is called \emph{clopen} if, for each $V\in\VN$, the component $\ps S_V$ is a clopen subset of $\SigN_V$. The set of clopen subobjects is denoted $\Subcl\SigN$. There is a partial order on $\Subcl\SigN$, given by
\begin{equation}
			\forall\ps S_1,\ps S_2\in\Subcl\SigN:\ps S_1\leq\ps S_2 \quad :\Longleftrightarrow\quad (\forall V\in\VN:\ps S_{1;V}\subseteq\ps S_{2;V}).
\end{equation}
\end{definition}
With respect to this order, $\Subcl\SigN$ is a complete distributive lattice. It was shown in \cite{DoeIsh08b} that $\Subcl\SigN$ is a (complete) Heyting algebra, and in \cite{Doe12} that it is a complete bi-Heyting algebra. All the subobjects of the form $\ps\deo(\hP_t)$ that arise from daseinisation are clopen subobjects, and hence the algebra $\Subcl\SigN$ is interpreted as an algebra of propositions. $\Subcl\SigN$ plays a central role in the new form of logic for quantum systems arising from the topos approach.

We will define the action of the unitary group $\UN$ on $\Subcl\SigN$ and time evolution in the Heisenberg picture in terms of flows on $\Subcl\SigN$, before coming back to elementary propositions of the form ``$\Ain\de;t$'' and their representing subobjects $\ps\deo(\hP_t)$ at the end of this subsection.

\begin{remark}
The complete bi-Heyting algebra $\Subcl\SigN$ can also be interpreted as the analogue of the complete Boolean algebra of measurable subsets modulo null subsets of a measure space. In fact, using $\SigN$ as a sample space for a quantum system, with $\Subcl\SigN$ as the algebra of measurable subsets, all basic structures of quantum probability can be formulated in a way that is structurally completely analogous to classical probability, see \cite{DoeDew12b} (and \cite{DoeDew12a}). We will make use of this interpretation of $\Subcl\SigN$ in subsection \ref{Subsec_SchroedingerPic} when representing quantum states as probability measures on $\SigN$, with the elements of $\Subcl\SigN$ as measurable sub`sets' (or rather, subobjects).
\end{remark}

Let $\pio=\pair{\Ga}{\io}$ be an automorphism of $\SigN$ with underlying base map $\ga:\VN\ra\VN$, and let $\ps S\in\Subcl\SigN$ be a clopen subobject. Then
\begin{equation}
			\forall V\in\VN: (\Ga^*(\ps S))_V=\ps S_{\ga(V)}
\end{equation}
and, by restricting $\io_V:\SigN_{\ga(V)}\ra\SigN_V$ to $(\Ga^*(\ps S))_V$, we have $\io_V((\Ga^*(\ps S))_V)\subseteq\SigN_V$ for all $V\in\VN$. Moreover, if $V'\subset V$, then
\begin{equation}
			(\Ga^*(\ps S))_{V'}=\ps S_{\ga(V')}\subseteq\ps S_{\ga(V)},
\end{equation}
because $\ps S$ is a subobject. Furthermore,
\begin{equation}
			\io_{V'}((\Ga^*(\ps S))_{V'})\subseteq\io_V(\ps S_{\ga(V)}),
\end{equation}
so the components $\io_V(\ps S_{\ga(V)})$, $V\in\VN$, form a subobject of $\Sig$, which we denote $\pio(\ps S)$. Clearly, this subobject is also clopen, so each automorphism $\pio=\pair{\Ga}{\io}$ of $\SigN$ induces a bijection
\begin{align}
			\ptio:\Subcl\SigN &\lra \Subcl\SigN\\			\nonumber
			\ps S &\lmt \io(\Ga^*(\ps S)).
\end{align}

Let $\ps S_1,\ps S_2\in\Subcl\SigN$ be two clopen subobjects such that $\ps S_1\leq\ps S_2$. Then, for all $V\in\VN$,
\begin{align}
			(\ptio(\ps S_1))_V &= \io_V(\Ga^*(\ps S_1))\\
			&= \io_V(\ps S_{1;\ga(V)})\\
			&\subseteq \io_V(\ps S_{2;\ga(V)})\\
			&= (\ptio(\ps S_2))_V,
\end{align}
so an automorphism $\ptio$ of $\Subcl\SigN$ preserves the order. $\ptio$ has an inverse, so it also reflects the order and hence is an order-automorphism. This implies that $\ptio$ preserves the bi-Heyting structure on $\Subcl\SigN$.

If we consider an automorphism $\pio = \pair{\Ga}{\io}:\SigN\ra\SigN$ of the spectral presheaf as the analogue of a measurable function, it is natural to use the inverse image (and not the forward image, as we did so far) of this `function' to act on $\Subcl\SigN$, which is the analogue of the algebra of measurable subsets.

We assume from now on that $\cN$ is not isomorphic to $\bbC^2$ and has no type $I_2$ summand. Since $\pio = \pair{\Ga}{\io}:\SigN\ra\SigN$ is an automorphism, induced by a unique base map $\ga:\VN\ra\VN$ by Thm. \ref{Thm_GroupIsos}, the inverse image on clopen subobjects is determined by the inverse transformation $\pio^{-1} = \pair{\Ga}{\io}^{-1}:\SigN\ra\SigN$ (on the spectral presheaf itself), which is the automorphism induced by the base map $\ga^{-1}:\VN\ra\VN$. Let $T:\cN\ra\cN$ denote the Jordan $*$-automorphism corresponding to $\pio$ by Thm. \ref{Thm_GroupIsos}.
Clearly, $T^{-1}:\cN\ra\cN$ is the Jordan $*$-automorphism corresponding to $\pio^{-1}$. Moreover,
\begin{equation}
			\pio^{-1}=\pair{\Ga}{\io}^{-1}=\pair{\Ga^{-1}}{\io^{-1}},
\end{equation}
where $\Ga^{-1}:\SetVNop\ra\SetVNop$ is the essential geometric automorphism induced by $\ga^{-1}:\VN\ra\VN$, and $\io^{-1}:(\Ga^{-1})^*(\SigN)\ra\SigN$ is the natural isomorphism with components, for all $V\in\VN$,
\begin{align}			\label{Eq_ComponentOfeta-1}
			\io^{-1}_V:((\Ga^{-1})^*(\SigN))_V=\SigN_{\ga^{-1}(V)} &\lra \SigN_V\\			\nonumber
			\ld &\lra \ld \circ T^{-1}|_V,
\end{align}
that is, precomposition with $T^{-1}|_V:V\ra\ga^{-1}(V)$, the restriction of the Jordan $*$-automorphism $T^{-1}$ to $V\in\VN$.

We will write
\begin{align}
			\ptio^{-1}:\Subcl\SigN &\lra \Subcl\SigN\\			\nonumber
			\ps S &\lra \io^{-1}((\Ga^{-1})^*(\ps S))
\end{align}
for the automorphism of $\Subcl\SigN$ induced by $\pio^{-1}=\pair{\Ga^{-1}}{\io^{-1}}$. Note that the assignment $\pio\mt\ptio^{-1}$ is contravariant. By Thm. \ref{Thm_GroupIsos}, the assignment $\pio\mt T$ of the Jordan $*$-automorphism $T$ of $\cN$ to an automorphism $\pio=\pair{\Ga}{\io}$ of $\SigN$ is also contravariant, so we have shown:

\begin{proposition}			\label{Prop_InjGroupHomomAut(SigN)ToAut(SubclSigN)}
Let $\cN$ be a von Neumann algebra. There is an injective group homomorphism
\begin{align}
			\Aut(\SigN) &\lra \Aut_{biHeyt}(\Subcl\SigN)^{\op}\\			\nonumber
			\pio=\pair{\Ga}{\io} &\lmt \ptio^{-1}
\end{align}
from the automorphism group of $\SigN$ into the opposite of the automorphism group of the complete bi-Heyting algebra $\Subcl\SigN$ of clopen subobjects. If $\cN$ is not isomorphic to $\bbC^2$ and has no type $I_2$ summand, there is an injective group homomorphism
\begin{equation}
			\Aut_{Jordan}(\cN) \lra \Aut_{biHeyt}(\Subcl\SigN)
\end{equation}
from the group of Jordan $*$-automorphisms of $\cN$ into the group of automorphisms of the complete bi-Heyting algebra $\Subcl\SigN$.
\end{proposition}

\begin{corollary}
Let $\cN$ be a von Neumann algebra not isomorphic to $\bbC^2$ and with no type $I_2$ summand. Let $\UN$ be the unitary group of $\cN$, let $\UN_0$ be the intersection of $\UN$ with the center of $\cN$, and let $\UN_{\on{proper}}:=\UN/\UN_0.$  There is an injective group homomorphism from $\UN_{\on{proper}}$ into the opposite of the automorphism group of the complete bi-Heyting algebra $\Subcl\SigN$.
\end{corollary}

\begin{proof}
Every $[\hU]\in\UN_{\on{proper}}$ induces an inner algebra automorphism $\phi_{\hU}:\cN\ra\cN$, where $\hU$ is a representative of $[\hU]$, and the map $[\hU]\ra\phi_{\hU}$ is injective. $\phi_{\hU}$ is an algebra automorphism and hence gives a Jordan $*$-automorphism, so $\UN_{\on{proper}}$ can be considered as a subgroup of $\Aut_{Jordan}(\cN)$. Then Prop. \ref{Prop_InjGroupHomomAut(SigN)ToAut(SubclSigN)}, second statement, applies.
\end{proof}

\begin{lemma}			\label{Lem_RotatingProjsInSubobjs}
Let $\cN$ be a von Neumann algebra with no type $I_2$ summand, and let $\SigN$ be its spectral presheaf. Let $\pio = \pair{\Ga}{\io}$ be an automorphism of $\SigN$ with underlying base map $\ga:\VN\ra\VN$, and let $\ptio^{-1}:\Subcl\SigN\ra\Subcl\SigN$ be the induced automorphism of $\Subcl\SigN$ (with base map $\ga^{-1}$). For $\ps S\in\Subcl\SigN$, let
\begin{align}
			&\hP_{\ptio^{-1}(\ps S)_V} := \alpha_V^{-1}(\ptio^{-1}(\ps S)_V),\\
			&\hP_{\ps S_{\ga^{-1}(V)}}:=\alpha_{\ga(V)}^{-1}(\ps S_{\ga^{-1}(V)})=\alpha_{\ga(V)}^{-1}(((\Ga^{-1})^*(\ps S))_V).
\end{align}
Then, for all $V\in\VN$,
\begin{equation}			\label{Eq_ProjsOfSubobjsUnderAutoms}
			\hP_{\ptio^{-1}(\ps S)_V} = T|_{\ga^{-1}(V)}(\hP_{\ps S_{\ga^{-1}(V)}}),
\end{equation}
where $T:\cN\ra\cN$ is the Jordan $*$-automorphism of $\cN$ induced by $\pio = \pair{\Ga}{\io}$ (cf. Thm. \ref{Thm_GroupIsos}).
\end{lemma}

\begin{proof}
Let $T:\cN\ra\cN$ be the Jordan $*$-automorphism of $\cN$ corresponding to the automorphism $\pio:\SigN\ra\SigN$, and let $T^{-1}$ be the inverse Jordan $*$-automorphism corresponding to $\pio^{-1}$. Then, by eq. \eq{Eq_ComponentOfeta-1}, the homeomorphism $\io^{-1}_V:((\Ga^{-1})^*(\SigN))_V=\SigN_{\ga^{-1}(V)}\ra\SigN_V$ is given by $\io^{-1}_V(\ld)=\ld\circ T^{-1}|_V$ for all $\ld\in\SigN_{\ga(V)}$. 

Noting that the inverse of $T^{-1}|_V:V\ra\ga^{-1}(V)$ is $T|_{\ga^{-1}(V)}:\ga^{-1}(V)\ra V$, we have
\begin{align}
			\ld\in((\Ga^{-1})^*(\ps S))_V=\ps S_{\ga^{-1}(V)} &\Longleftrightarrow \ld(\hP_{\ps S_{\ga^{-1}(V)}}) = 1\\
			&\Longleftrightarrow (\ld\circ T^{-1}|_V\circ T|_{\ga{-1}(V)})(\hP_{\ps S_{\ga^{-1}(V)}})=1\\
			&\Longleftrightarrow (\ld\circ T^{-1}|_V)(T|_{\ga^{-1}(V)}(\hP_{\ps S_{\ga^{-1}(V)}}))=1\\
			&\Longleftrightarrow \io^{-1}_V(\ld)(T|_{\ga^{-1}(V)}(\hP_{\ps S_{\ga^{-1}(V)}}))=1,
\end{align}
and also
\begin{align}
			\ld\in((\Ga^{-1})^*(\ps S))_V &\Longleftrightarrow \io^{-1}_V(\ld)\in\io^{-1}_V(((\Ga^{-1})^*(\ps S))_V)\\
			&\Longleftrightarrow \io^{-1}_V(\ld)(\hP_{(\ptio^{-1}(\ps S))_V})=1,
\end{align}
so $\io^{-1}_V(\ld)(T|_{\ga^{-1}(V)}(\hP_{\ps S_{\ga^{-1}(V)}}))=1$ if and only if $\io^{-1}_V(\ld)(\hP_{(\ptio^{-1}(\ps S))_V})=1$, and hence
\begin{equation}
			\hP_{(\ptio^{-1}(\ps S))_V} = T|_{\ga^{-1}(V)}(\hP_{\ps S_{\ga^{-1}(V)}}).
\end{equation}
\end{proof}

This lemma shows that the projections corresponding to the components of the subobject $\ptio^{-1}(\ps S)$ are given by the projections corresponding to the components of the original subobject $\ps S$, `rotated' by $T$, where $T$ is the Jordan $*$-automorphism corresponding to $\pio$. Note that Jordan automorphisms correspond \emph{contra}variantly to automorphisms of the spectral presheaf by Thm. \ref{Thm_GroupIsos}. Hence, $T$ goes in the `opposite direction' relative to $\pio$, and hence in the `same direction' as $\ptio^{-1}:\Subcl\SigN\ra\Subcl\SigN$.

\begin{corollary}			\label{Cor_UnitaryActionOnProjs}
Let $\cN$ be a von Neumann algebra with no type $I_2$ summand, let $\hU\in\UN$ and $\phi_{\hU}:\cN\ra\cN$, $\hA\mt\hU\hA\hU^*$. If $\pphi_{\hU}:\SigN\ra\SigN$ is the automorphism induced by $\phi_{\hU}$ and $\ptphi_{\hU^*}:\Subcl\SigN\ra\Subcl\SigN$ is the corresponding automorphism of $\Subcl\SigN$ (induced by $\phi_{\hU}^{-1}=\phi_{\hU^*}$), then, for all $\ps S\in\Subcl\SigN$ and for all $V\in\VN$,
\begin{equation}
			\hP_{\ptphi_{\hU^*}(\ps S)_V} = \hU\hP_{\ps S_{\hU^* V\hU}}\hU^*,
\end{equation}
where the projections are given by $\hP_{\ptphi_{\hU^*}(\ps S)_V} = \alpha_V^{-1}(\ptphi_{\hU^*}(\ps S)_V)$ and $\hP_{\ps S_{\hU^* V\hU}} = \alpha_{\hU^* V\hU}^{-1}(\ps S_{\hU^* V\hU})$ (cf. \eq{Def_alphaV}).
\end{corollary}



\begin{definition}			\label{Def_FlowOnSubclSigN}
Let $F:\bbR\ra\Aut(\SigN)$ be a flow on the spectral presheaf of $\cN$. The \emph{flow $\tilde F^{-1}:\bbR\ra\Aut_{biHeyt}(\Subcl\SigN)$ on clopen subobjects corresponding to $F$} is the one-parameter group $(\tilde F^{-1}(t))_{t\in\bbR}:\Subcl\SigN\ra\Subcl\SigN$ of automorphisms of the complete bi-Heyting algebra $\Subcl\SigN$, given as follows: for all $t\in\bbR$, if $F(t)=\pio_t=\pair{\Ga_t}{\io_t}$, then
\begin{align}
			\tilde F^{-1}(t) = \ptio_t^{-1}:\Subcl\SigN &\lra \Subcl\SigN\\			\nonumber
			\ps S &\lmt \io_t^{-1}((\Ga_t^{-1})^*(\ps S)).
\end{align}
Let $(\hU_t)_{t\in\bbR}=(e^{it\hH})_{t\in\bbR}$ be the one-parameter group of unitaries in $\cN$ induced by a self-adjoint operator $\hH$ affiliated with $\cN$. We define the corresponding flow $(\tilde F^{-1}_{\hH}(t))_{t\in\bbR}$ on $\Subcl\SigN$, for all $t\in\bbR$, by
\begin{equation}
			\tilde F^{-1}_{\hH}(t) := \ptphi_{\hU_t}^{-1}:\Subcl\SigN \lra \Subcl\SigN
\end{equation}
such that $\tilde F^{-1}_{\hH}(\ps S) = \cG_{\hU_{t}^{-1}}((\Phi_{\hU_{t}^{-1}})^*(\ps S))$.
\end{definition}

If $F:\bbR\ra\Aut(\SigN)$, $t\ra\pio_t=\pair{\Ga_t}{\io_t}$ is a flow and $\tilde F^{-1}:\bbR\ra\Aut_{biHeyt}(\Subcl\SigN)$ is the corresponding flow on clopen subobjects, we have $\io_t^{-1}=\io_{-t}$ and $\Ga_t^{-1}=\Ga_{-t}$ and hence $\ptio_t^{-1}=\ptio_{-t}$ for all $t\in\bbR$, so
\begin{align}
			\tilde F^{-1}(t) = \ptio_{-t}:\Subcl\SigN\lra\Subcl\SigN.
\end{align}
If $F_{\hH}:\bbR\ra\Aut(\SigN)$ is the flow induced by a one-parameter group of unitaries, then $\hU_t^{-1}=\hU_{-t}$ for all $t\in\bbR$, so
\begin{align}			\label{Eq_FlowForwardOnClopenSubojs}
	\tilde F^{-1}_{\hH}(t) = \ptphi_{\hU_{-t}}:\Subcl\SigN\lra\Subcl\SigN
\end{align}
is the induced flow on clopen subobjects.


We now come back to propositions about the values of physical quantities, and how they change in time. Let ``$\Ain\de;t_0$'' be such a proposition, and let $\hP_{t_0}$ the corresponding projection in $\cN$ (via the spectral theorem). Without loss of generality, we can assume $t_0=0$. We saw in \eq{Def_OuterDasOfProj} that there is a clopen subobject $\ps\deo(\hP_0)$ of $\SigN$ representing the proposition ``$\Ain\de;0$''.

In standard quantum theory in the Heisenberg picture, time evolution is implemented by a (strongly continuous) one-parameter group $(\hU_t^*)_{t\in\bbR}$ of unitaries in $\cN$ acting on self-adjoint operators (representing physical quantities), and hence also acting on projections (representing propositions). For any given $t\in\bbR$, the proposition ``$\Ain\de;t$'' is represented by the projection operator $\hP_t:=\hU_{-t}\hP_0\hU_t$.\footnote{\label{FN_TimeEvolutionOfProjs}Note that $\hP_0$ is conjugated by $\hU_{-t}=\hU_t^*$ (and not by $\hU_t$). This is the usual convention in quantum theory, which we follow from now on.} This determines a clopen subobject $\ps\deo(\hP_t)=\ps\deo(\hU_{-t}\hP_0\hU_t)$. We will now show that the flow on $\Subcl\SigN$ associated with the one-parameter group $(\hU_{-t})_{t\in\bbR}$ actually maps $\ps\deo(\hP_0)$ to $\ps\deo(\hP_t)$, for each $t\in\bbR$.

\begin{proposition}
Let $\cN$ be a von Neumann algebra with no type $I_2$ summand, and let $(\hU_{-t})_{t\in\bbR}$ be a one-parameter group of unitaries in $\cN$, induced by a self-adjoint operator $\hH$ affiliated with $\cN$ such that $\hU_{-t}=e^{-it\hH}$. Let $(\tilde F_{\hH}^{-1}(t))_{t\in\bbR}$ be the flow on $\Subcl\SigN$ given by
\begin{equation}
			\tilde F_{\hH}^{-1}(t) = \ptphi_{\hU_{-t}}^{-1} = \ptphi_{\hU_{t}}
\end{equation}
for all $t\in\bbR$.\footnote{Note that this convention differs from eq. \eq{Eq_FlowForwardOnClopenSubojs} because we consider the one-parameter group $(\hU_{-t})_{t\in\bbR}$ rather than $(\hU_t)_{t\in\bbR}$, cf. footnote \ref{FN_TimeEvolutionOfProjs}.} Moreover, let ``$\Ain\de;0$'' be a proposition, represented by the projection $\hP_0\in\PN$, and let $\hP_t=\hU_{-t}\hP_0\hU_t$ be the projection in $\cN$ that represents the proposition ``$\Ain\de;t$'' at time $t$. Let $\ps\deo(\hP_0)\in\Subcl\SigN$ be the clopen subobject corresponding to $\hP_0$ by eq. \eq{Def_OuterDasOfProj}, and let $\ps\deo(\hP_t)$ be the clopen subobject corresponding to $\hP_t$. Then, for all $t\in\bbR$,
\begin{equation}
			\tilde F_{\hH}^{-1}(t)(\ps\deo(\hP_0)) = \ps\deo(\hP_t).
\end{equation}
\end{proposition}

\begin{proof}
For all $V\in\VN$, we have
\begin{equation}
			\hP_{(\tilde F_{\hH}^{-1}(t)(\ps\deo(\hP_0)))_V} 
			= \hP_{\ptphi_{\hU_{t}}(\ps\deo(\hP_0))_V}
			\stackrel{\text{Cor. }\ref{Cor_UnitaryActionOnProjs}}{=} \hU_{-t}\hP_{\ps\deo(\hP_0)_{\hU_t V\hU_{-t}}}\hU_t.
\end{equation}
Moreover, for all $V\in\VN$,
\begin{align}
			&\hU_{-t}\hP_{\ps\deo(\hP_0)_{\hU_t V\hU_{-t}}}\hU_t\nonumber\\
			 \stackrel{\eq{Def_OuterDasOfProjs}}{=}\;&\hU_{-t}\bmeet\{\hU_t\hQ\hU_{-t}\in\mc P(\hU_t V\hU_{-t}) \mid \hU_t\hQ\hU_{-t}\geq\hP_0)\}\hU_t\\
			=\;&\bmeet\{\hQ\in\PV \mid \hU_t\hQ\hU_{-t}\geq\hP_0\}\\
			=\;&\bmeet\{\hQ\in\PV \mid \hQ\geq\hU_{-t}\hP_0\hU_t\}\\
			\stackrel{\eq{Def_OuterDasOfProjs}}{=}\;&\deo(\hU_{-t}\hP_0\hU_t)_V\\
			=\;&\hP_{(\ps\deo(\hU_{-t}\hP_0\hU_t))_V},
\end{align}
so $\hP_{(\tilde F_{\hH}^{-1}(t)(\ps\deo(\hP_0)))_V} = \hP_{(\ps\deo(\hU_{-t}\hP_0\hU_t))_V}$ for all $V\in\VN$, and hence
\begin{equation}
			\tilde F^{-1}(t)(\ps\deo(\hP_0)) = \ps\deo(\hU_{-t}\hP_0\hU_t)=\ps\deo(\hP_t).
\end{equation}
\end{proof}

More generally, if $\ps S_0\in\Subcl\SigN$ is a clopen subobject of the spectral presheaf that represents some proposition about the quantum system at time $t_0=0$, then time evolution (for time $t$) in the Heisenberg picture maps $\ps S_0$ to
\begin{equation}			\label{Def_S_tHeisenberg}
			\ps S_t := \tilde F^{-1}(t)(\ps S_0) = \ptphi_{\hU_t}(\ps S_0).
\end{equation}

\subsection{States as probability measures and Schr\"odinger picture}			\label{Subsec_SchroedingerPic}
Quantum states are identified with states of the von Neumann algebra $\cN$. In order to describe time evolution in the Schr\"odinger picture, we need a representation of states as structures related to the spectral presheaf.

\begin{definition}			\label{Def_ProbabMeasOnClopenSubobjs}
(\cite{Doe09a}) Let $\cN$ be a von Neumann algebra with context category $\VN$, let $\Subcl\SigN$ be the complete bi-Heyting algebra of clopen subobjects of $\SigN$, and let $A(\VN,[0,1])$ be the set of antitone maps from the poset $\VN$ to the unit interval. A map
\begin{equation}
			\mu:\Subcl\SigN \lra A(\VN,[0,1])
\end{equation}
is called a \emph{probability measure on (the clopen subobjects of) the spectral presheaf $\SigN$} if
\begin{itemize}
	\item [1.] $\mu(\Sig)=1_{\VN}$, the constant function on $\VN$ with value $1$,
	\item [2.] for all $\ps S,\ps T\in\Subcl\SigN$,
	\begin{equation}
				\mu(\ps S)+\mu(\ps T)=\mu(\ps S\join\ps T)+\mu(\ps S\meet\ps T),
	\end{equation}
	where meets, joins and addition are defined stagewise, at each $V\in\VN$ locally (e.g. $(\mu(\ps S)+\mu(\ps T))(V)=\mu(\ps S)(V)+\mu(\ps T)(V)$).\footnote{In \cite{Doe09a}, the codomain of a measure was given equivalently as the set $\Ga\ps{[0,1]^\succeq}$ of global sections of a certain presheaf.}
\end{itemize}
The convex set of probability measures on $\SigN$ is denoted $\mc M(\SigN)$.
\end{definition}

The following result, which was proven in \cite{Doe09a}, shows that probability measures on $\SigN$ in the sense defined above indeed correspond to states of the von Neumann algebra $\cN$:

\begin{theorem}			\label{Thm_IsomS(N)AndM(SigM)}
Let $\cN$ be a von Neumann algebra with no type $I_2$ summand. There is an isomorphism of convex sets between $\mc S(\cN)$, the convex set of states on $\cN$, and $\mc M(\SigN)$, the convex set of probability measures on the spectral presheaf of $\cN$.
\end{theorem}

The proof idea is to show that each probability measure $\mu:\Subcl\SigN\ra\cA(\VN,[0,1])$ determines a unique finitely additive probability measure $m:\PN\ra [0,1]$ on the projections in the von Neumann algebra $\cN$. By the generalised version of Gleason's theorem (see \cite{Mae90} and references therein), such a map $m$ induces a unique state $\rho_\mu$ of $\cN$. Conversely, every state $\rho$ gives a probability measure $\mu_\rho$ on $\SigN$. Normal states were characterised in \cite{Doe09a} by a local property, and by a regularity condition in \cite{DoeIsh12}: normal states correspond bijectively to those probability measures $\mu:\Subcl\SigN\ra A(\VN,[0,1])$ that preserve joins of increasing families of clopen subobjects. (If the von Neumann algebra $\cN$ can be represented faithfully on a separable Hilbert space, then preservation of countable joins is sufficient.)

We will now give another, equivalent characterisation of states as probability measures on the spectral presheaf $\SigN$. This description makes it easier to define the action of one-parameter groups of unitaries on the set of measures, which physically corresponds to time evolution in the Schr\"odinger picture of the quantum system described by $\cN$.

Let $\rho\in\mc S(\cN)$ be a state of $\cN$. For each abelian von Neumann subalgebra $V\in\VN$, we obtain a finitely additive probability measure (\emph{FAPM})
\begin{align}
			\rho|_V:\PV &\lra [0,1]\\			\nonumber
			\hP &\lmt \rho(\hP)
\end{align}
on the projections in $V$, that is,
\begin{equation}
			\forall \hP,\hQ\in\PV: \hP\hQ=\hat 0 \Longleftrightarrow \rho|_V(\hP+\hQ)=\rho|_V(\hP)+\rho|_V(\hQ)
\end{equation}
and $\rho|_V(\hat 1)=1$. Using the isomorphism
\begin{align}
			\alpha_V:\PV &\lra \Cl(\SigN_V)\\			\nonumber
			\hP &\lmt \{\ld\in\SigN_V \mid \ld(\hP)=1\}
\end{align}
between the complete Boolean algebras of projections in $V$ and clopen subsets of the Gelfand spectrum of $V$, we can think of $\rho|_V$ as a FAPM
\begin{equation}
			\rho|_V\circ\alpha_V^{-1}:\Cl(\SigN_V) \lra [0,1]
\end{equation}
on clopen subsets of the Gelfand spectrum $\SigN_V$ of $V$. Note that $\Cl(\SigN_V)=(\Subcl\SigN)_V$ for all $V\in\VN$.

\begin{definition}
Let $\cN$ be a von Neumann algebra, and let $\VN$ be its context category. The presheaf $\CP$ of classical probability measures on $\SigN$ is given
\begin{itemize}
	\item [(a)] on objects: for all $V\in\VN$,
	\begin{equation}
				\CP_V:=\{m_V:\Cl(\SigN_V)\ra [0,1] \mid m_V\text{ is a FAPM}\},
	\end{equation}
	that is, $m_V(\SigN_V)=1$ and, for all $S_1,S_2\in\Cl(\SigN_V)$,
	\begin{align}			\label{Eq_FAPM}
					m_V(S_1)+m_V(S_2) = m_V(S_1\cup S_2)+m_V(S_1\cap S_2)
	\end{align}
	for all $m_V\in\CP_V$.
	\item [(b)] on arrows: for all inclusions $i_{V'V}$,
	\begin{align}
				\CP(i_{V'V}):\CP_V &\lra \CP_{V'}\\			\nonumber
				m_V &\lmt m_V\circ\Sig(i_{V'V})^{-1},
	\end{align}
	that is, $\CP(i_{V'V})(m_V)$ is the pushforward of the measure of $m_V$ along $\Sig(i_{V'V})$.
\end{itemize}
\end{definition}
Here, we use the fact that the restriction map $\SigN(i_{V'V})$ is continuous and hence measurable (i.e, $\SigN(i_{V'V})^{-1}$ takes measurable subsets of $\SigN_{V'}$ to measurable subsets of $\SigN_V$).

\begin{proposition}			\label{Prop_IsomS(N)AndGaCP}
Let $\cN$ be a von Neumann algebra with no type $I_2$ summand. There is an isomorphism of convex sets between $\mc S(\cN)$, the set of states of $\cN$, and the set $\Ga\CP$ of global sections of the presheaf $\CP$.
\end{proposition}

\begin{proof}
Let $\rho\in\mc S(\cN)$. Then the family $(\rho|_V\circ\alpha_V^{-1})_{V\in\VN}$, where
\begin{align}
			\rho|_V\circ\alpha_V^{-1}:\Cl(\SigN_V)\ra [0,1]
\end{align}
for each $V$, clearly is a global section of $\CP$.

Conversely, let $m=(m_V)_{V\in\VN}$ be a global section of $\CP$. For each $V\in\VN$, we have a FAPM
\begin{equation}
			m_V:\Cl(\SigN_V) \lra [0,1]
\end{equation}
and hence a FAPM $m_V\circ\alpha_V:\PV\ra [0,1]$ by the isomorphism $\alpha_V:\PV\ra\Cl(\SigN_V)$. The fact that $m=(m_V)_{V\in\VN}$ is a global section of $\CP$ implies that
\begin{equation}			\label{Eq_RestrictionOfProbabMeas}
			(m_V\circ\alpha_V)|_{\mc P(V')}=m_{V'}\circ\alpha_{V'}
\end{equation}
for all $V,V'\in\VN$ such that $V'\subset V$. For all $\hP\in\PN$, let
\begin{equation}
			\mu(\hP):=(m_V\circ\alpha_V)(\hP),
\end{equation}
where $V$ is some abelian von Neumann subalgebra that contains $\hP$. This is well-defined because of eq. \eq{Eq_RestrictionOfProbabMeas}. Clearly, $\mu:\PN\ra [0,1]$ is a finitely additive probability measure on the projections in $\cN$. By the generalised version of Gleason's theorem \cite{Mae90}, the measure $\mu$ corresponds to a unique state $\rho_\mu$ on $\cN$ such that $\rho_\mu|_{\PN}=\mu$.
\end{proof}

\begin{lemma}			\label{Lem_GaCPAndM(SigN)}
There is an isomorphism of convex sets between the set $\Ga\CP$ of global sections of the presheaf $\CP$ and the set $\cM(\SigN)$ of probability measures on $\SigN$ (as in Def. \ref{Def_ProbabMeasOnClopenSubobjs}).
\end{lemma}

\begin{proof}
Let $m=(m_V)_{V\in\VN}$ be a global section of $\Ga\CP$. This defines a map
\begin{align}
			\tilde m:\Subcl\SigN &\lra A(\VN,[0,1])\\			\nonumber
			\ps S=(\ps S_V)_{V\in\VN} &\lmt (m_V(\ps S_V))_{V\in\VN}
\end{align}
such that $\tilde m(\SigN)=(m_V(\SigN_V))_{V\in\VN}=1_{\VN}$. Moreover, writing $\hP_S:=\alpha_V^{-1}(S)$, where $S\in\Cl(\SigN_V)$, we have for all $\ps S_1,\ps S_2\in\Subcl\SigN$ and for all $V\in\VN$,
\begin{align}
			(\tilde m(\ps S_1)+\tilde m(\ps S_2))_V &= m_V(\ps S_{1;V})+m_V(\ps S_{2;V})\\
			&\stackrel{\eq{Eq_FAPM}}{=} m_V(\ps S_{1;V}\cup\ps S_{2;V})+m_V(\ps S_{1;V}\cap\ps S_{1;V})\\
			&= (\tilde m(\ps S_1\join\ps S_2))_V+(\tilde m(\ps S_1\meet\ps S_2))_V,
\end{align}
so $\tilde m\in\mc M(\SigN)$. 

Conversely, let $\mu\in\mc M(\SigN)$. If $S\in\Cl(\SigN_V)$ is a clopen subset of $\SigN_V$, the Gelfand spectrum of $V$, then there is a clopen subobject $\ps S\in\Subcl\SigN$ such that $\ps S_V=S$. Define
\begin{align}
			\mu_V:\Cl(\SigN_V) &\lra [0,1]\\			\nonumber
			S &\lmt \mu(\ps S)(V).
\end{align}
For each $V\in\VN$, $\mu_V(\SigN_V)=\mu(\SigN_V)(V)=1$. Moreover, if $S,T\in\Cl(\SigN_V)$ are disjoint clopen subsects of $\SigN_V$, there are clopen subobjects of $\ps S,\ps T$ of $\SigN$ such that $\ps S_V=S$ and $\ps T_V=T$, and, for all $V\in\VN$,
\begin{align}
			\mu_V(S\cup T) &= \mu_V((\ps S\join\ps T)_V)\\
			&= \mu(\ps S\join\ps T)(V)\\
			&= (\mu(\ps S)+\mu(\ps T)-\mu(\ps S\meet\ps T))(V)\\
			&= \mu(\ps S)(V)+\mu(\ps T)(V)\\
			&= \mu_V(S)+\mu_V(T),
\end{align}
so $\mu_V$ is a finitely additive probability measure on $\Cl(\SigN_V)$, and 
\begin{equation}
			\tilde\mu:=(\mu_V)_{V\in\VN}\in\Ga\CP. 
\end{equation}
Clearly, the two maps are inverse to each other.
\end{proof}

Prop. \ref{Prop_IsomS(N)AndGaCP} and Lemma \ref{Lem_GaCPAndM(SigN)} together immediately imply Thm. \ref{Thm_IsomS(N)AndM(SigM)} (so we have an alternative proof to the one given in \cite{Doe09a}). The workhorse here is of course the generalised version of Gleason's theorem \cite{Mae90}. Prop. \ref{Prop_IsomS(N)AndGaCP} can be read as a reformulation of this powerful theorem in terms of global sections of a certain presheaf. This is conceptually similar to the reformulation of the Kochen-Specker theorem in terms of non-existence of global sections of the spectral presheaf \cite{IshBut98,IHB00,Doe05}. De Groote formulated a similar result in \cite{deG07}, Thm. 7.2.

We will write $m_\rho=(m_{\rho;V})_{V\in\VN}$ for the global section of $\CP$ that corresponds to a given state $\rho$ of $\cN$. The components of $m_\rho$ are given by
\begin{align}			\label{Eq_ComponentsOfm_rho}
			\forall V\in\VN: m_{\rho;V} = \rho|_V\circ\alpha_V^{-1}:\Cl(\SigN_V)\lra [0,1].
\end{align}
Here, the isomorphism $\alpha_V^{-1}$ merely switches from clopen subsets of $\SigN_V$ to projections in $V$.

\begin{definition}			\label{Def_StatePropPairing}
There is a \emph{state-proposition pairing}
\begin{align}
			p:\Ga\CP\times\Subcl\SigN &\lra \mc A(\VN,[0,1])\\			\nonumber
			(m_\rho,\ps S) &\lmt m_\rho(\ps S),
\end{align}
given by
\begin{equation}
			\forall V\in\VN: (m_\rho(\ps S))_V := m_{\rho;V}(\ps S_V)
\end{equation}
\end{definition}

Let $\mu_\rho\in\mc M(\SigN)$ be the probability measure on $\SigN$ corresponding to $\rho$ by Thm. \ref{Thm_IsomS(N)AndM(SigM)}. We note that for each $V\in\VN$, we have $(m_\rho(\ps S))_V=\mu_{\rho;V}(\ps S_V)=(\mu_\rho(\ps S))_V$, so by construction,
\begin{equation}			\label{Eq_StatePropPairing}
			m_\rho(\ps S) = \mu_\rho(\ps S).
\end{equation}

In standard quantum theory, a unitary acts on a state in the following way: if $\rho$ is a normal state with density matrix $\trho$, that is, $\rho=\tr(\trho-):\cN\ra\bbC$ (where $\trho$ is followed by a `placeholder' into which an argument can be inserted), then
\begin{equation}			\label{Eq_TimeEvolutionOfStates}
			\hU.\rho=\hU.\tr(\trho-):=\tr(\hU\trho\hU^*-)=\tr(\trho\hU^*-\hU)=\rho\circ\phi_{\hU^*},
\end{equation}
where $\phi_{\hU^*}:\cN\ra\cN$ is the inner automorphism induced by $\hU^*$. More generally, even if $\rho$ is not a normal state, we define $\hU.\rho:=\rho\circ\phi_{\hU^*}$.

\begin{definition}
Let $m_\rho\in\Ga\CP$ be the global section of $\CP$ corresponding to the state $\rho$, let $\hU\in\UN$ be a unitary, and let $\phi_{\hU^*}$ be the inner automorphism induced by $\hU^*$. We define $\hU.m_\rho$ by
\begin{equation}			\label{Def_U.m_rho}
				(\hU.m_\rho)_V := m_{\rho;\hU^*V\hU}\circ\alpha_{\hU^* V\hU}\circ\phi|_{\hU^*}\circ\alpha_V^{-1}:\Cl(\SigN_V)\lra [0,1]
\end{equation}
for all $V\in\VN$.
\end{definition}

Here, the isomorphisms $\alpha_V^{-1}$ and $\alpha_{\hU^*V\hU}$ merely serve to switch from clopen subsets to projections and back.

\begin{lemma}			\label{Lem_StatePreshSchroedinger}
$\hU.m_\rho$ is the global section of $\CP$ corresponding to the state $\hU.\rho:=\rho\circ\phi_{\hU^*}$, that is, $\hU.m_\rho=m_{\hU.\rho}=m_{\rho\circ\phi_{\hU^*}}$.
\end{lemma}

\begin{proof}
First note that $m_{\rho;\hU^*V\hU}=\rho|_{\hU^*V\hU}\circ\alpha_{\hU^*V\hU}^{-1}$ for all $V\in\VN$ by eq. \eq{Eq_ComponentsOfm_rho}. Hence,
\begin{align}
			(\hU.m_\rho)_V &= \rho|_{\hU^*V\hU}\circ\alpha_{\hU^*V\hU}^{-1}\circ\alpha_{\hU^* V\hU}\circ\phi|_{\hU^*}\circ\alpha_V^{-1}\\
			&= \rho|_{\hU^*V\hU}\circ\phi|_{\hU^*}\circ\alpha_V^{-1}\\
			&= (\rho\circ\phi_{\hU^*})|_V\circ\alpha_V^{-1}\\
			&= m_{\hU.\rho;V}.
\end{align}
\end{proof}


\begin{definition}			\label{Def_StatePresheafAtTimet}
Let $(\hU_t)_{t\in\bbR}$ be a strongly continuous one-parameter group of unitaries in $\cN$, induced by a self-adjoint operator $\hA$ affiliated with $\cN$, and let $\rho_0\in\mc S(\cN)$ be a state. We write $\rho_t:=\hU_t.\rho_0=\rho_0\circ\phi_{\hU_{-t}}$ for the state evolved by time $t$ (we recall that $\hU_t^*=\hU_{-t}$), and $m_{\rho_t}=(m_{\rho_{t;V}})_{V\in\VN}$ for the corresponding global section of the presheaf $\CP$ of classical probability measures, that is,
\begin{align}
					m_{\rho_t} = \hU_t.m_{\rho_0} = m_{\hU_t.\rho_0}.
\end{align}
The \emph{flow on the global sections of the presheaf $\CP$ of classical probability measures} induced by the one-parameter group $(\hU_t)_{t\in\bbR}$ is the map
\begin{align}
			\ol F_{\hA}:\bbR\ra\Aut(\Ga\CP)\\
			t \lra m_{\rho_t}.
\end{align}
\end{definition}


By Lemma \ref{Lem_GaCPAndM(SigN)}, we have $\Ga\CP\simeq\mc M(\SigN)$, so we can regard a flow $\ol F_{\hA}:\bbR\ra\Aut(\Ga\CP)$ alternatively as a flow $\ol F_{\hA}:\bbR\ra\Aut(\mc M(\SigN))$.

\section{Compatibility between the Heisenberg picture and the Schr\"odinger picture and covariance under the unitary group}			\label{Sec_Compatibility}
In standard quantum theory, the Heisenberg picture and the Schr\"o- dinger picture of time evolution are compatible in the following sense: one can either apply time evolution to the proposition and leave the state fixed, or evolve the state and leave the proposition fixed -- the two situations cannot be distinguished physically. We will show that our reformulation of the Heisenberg picture and the Schr\"odinger picture based on flows on $\Subcl\SigN$ respectively on $\Ga\CP\simeq\mc M(\SigN)$ are compatible in an analogous manner.

Let $\rho_0\in\mc S(\cN)$, interpreted as the state of the quantum system at time $t_0=0$, let $\hP_0\in\PN$ be a projection, representing a proposition of the form ``$\Ain\de;0$'' at time $t_0=0$, and let $(\hU_t)_{t\in\bbR}$ be a strongly continuous one-parameter group of unitaries in $\cN$. Let $t\in\bbR$, let $\rho_t=\hU_t.\rho_0=\rho_0\circ\phi_{\hU_{-t}}$ be the state evolved by time $t$ (Schr\"odinger picture), and let $\hP_t=\hU_{-t}\hP_0\hU_t$ be the projection evolved by time $t$ (Heisenberg picture). Then
\begin{equation}
			\rho_t(\hP_0) = \rho_0\circ\phi_{\hU_{-t}}(\hP_0) =\rho_0(\hU_{-t}\hP_0\hU_t) = \rho_0(\hP_t).
\end{equation}
This is the compatibility between Schr\"odinger picture and Heisenberg picture of standard quantum theory. The expression $\rho_t(\hP_0)$ is the expectation value of the proposition ``$\Ain\de;0$'' represented by $\hP_0$ being true upon measurement, that is, the probability of finding a measurement outcome in $\de$ when $A$ is measured in the state $\rho_t$. Since we consider measurements of $\hP_0$ in the state $\rho_t$, we have to pick a context $V\in\VN$ (at least implicitly) that contains the binary observable $\hP_0$, because contexts represent co-measurable sets of physical quantities, and only these are accessible to experiment. Analogously, evaluating $\rho_0(\hP_t)$ means picking a context that contains $\hP_t$. If $V$ contains $\hP_0$, then $\hU_{-t}V\hU_t$ contains $\hP_t=\hU_{-t}\hP_0\hU_t$. The relevant contexts can be displayed explicitly:
\begin{equation}			\label{Eq_CompatibilityStandardForm}
			\rho_t|_V(\hP_0) = \rho_0|_{\hU_{-t}V\hU_t}(\hP_t).
\end{equation}

We now show that our definitions of the Heisenberg and the Schr\"o- dinger picture based on flows on $\Subcl\SigN$ respectively on $\Ga\CP$ mirror this compatibility property, including the dependence on two different contexts:

\begin{proposition}			\label{Prop_CompatibilityToposForm}
Let $\rho_0\in\mc S(\cN)$ be a state, and let $m_{\rho_0}$ be the corresponding global section of $\CP$ (see eq. \eq{Eq_ComponentsOfm_rho}). Let $\ps S_0\in\Subcl\SigN$ be a clopen subobject, representing a proposition about the quantum system, e.g. $\ps S_0=\ps\deo(\hP_0)$, where $\hP_0$ represents the proposition ``$\Ain\de;0$''. Let $(\hU_t)_{t\in\bbR}$ be a strongly continuous one-parameter group of unitaries in $\cN$, let $m_{\rho_t}$ be the global section of $\CP$ corresponding to $\rho_t=\hU_t.\rho_0=\rho_0\circ\phi_{\hU_{-t}}$ (cf. Def. \ref{Def_StatePresheafAtTimet}), and let $\ps S_t$ be the clopen subobject
\begin{equation}
			\ps S_t:=\ptphi_{\hU_t}(\ps S_0)
\end{equation}
(cf. eq. \eq{Def_S_tHeisenberg}). Then, for all $V\in\VN$ and all $t\in\bbR$,
\begin{equation}			\label{Eq_CompatibilityToposForm}
			(m_{\rho_t}(\ps S_0))_V = (m_{\rho_0}(\ps S_t))_{\hU_{-t}V\hU_t},
\end{equation}
where $m_{\rho_t}(\ps S_0)$ is the antitone function from $\VN$ to $[0,1]$ from Def. \eq{Def_StatePropPairing}, and analogously $m_{\rho_0}(\ps S_t)$.
\end{proposition}

\begin{proof}
We have, for all $V\in\VN$,
\begin{align}
			(m_{\rho_t}(\ps S_0))_V &\stackrel{\text{Def. \ref{Def_StatePropPairing}}}{=} m_{\rho_t;V}(\ps S_{0;V})\\
			&\stackrel{\text{Lemma \ref{Lem_StatePreshSchroedinger}}}{=} (\hU_t.m_{\rho_0})_V(\ps S_{0;V})\\
			&\stackrel{\eq{Def_U.m_rho}}{=} m_{\rho_0;\hU_{-t}V\hU_t}\circ\alpha_{\hU_{-t}V\hU_t}\circ\phi_{\hU_{-t}}\circ\alpha_V^{-1}(\ps S_{0;V})\\
			&= m_{\rho_0;\hU_{-t}V\hU_t}\circ\alpha_{\hU_{-t}V\hU_t}(\hU_{-t}\hP_{\ps S_{0;V}}\hU_t)\\
			&\stackrel{\text{Cor. \ref{Cor_UnitaryActionOnProjs}}}{=} m_{\rho_0;\hU_{-t}V\hU_t}\circ\alpha_{\hU_{-t}V\hU_t}(\hP_{\ptphi_{\hU_t}(\ps S_0)_{\hU_{-t}V\hU_t}})\\
			&= m_{\rho_0;\hU_{-t}V\hU_t}(\ptphi_{\hU_t}(\ps S_0)_{\hU_{-t}V\hU_t})\\
			&\stackrel{\eq{Def_S_tHeisenberg}}{=} m_{\rho_0;\hU_{-t}V\hU_t}(\ps S_{t;\hU_{-t}V\hU_t})\\
			&\stackrel{\text{Def. \ref{Def_StatePropPairing}}}{=} (m_{\rho_0}(\ps S_t))_{\hU_{-t}V\hU_t}.
\end{align}
\end{proof}


Equation \eq{Eq_CompatibilityToposForm} is the analogue of the usual quantum-theoretic relation \eq{Eq_CompatibilityStandardForm}. Yet, \eq{Eq_CompatibilityToposForm} holds for \emph{all} contexts $V\in\VN$ simultaneously. 

Information that can be extracted physically from a quantum system is (a) outcomes of experiments, and (b) expectation values of such outcomes in the long run. The well-known \emph{Born rule} allows to calculate the expectation value of a physical quantity in a given state. As was shown in \cite{Doe09a,DoeIsh12} (see also \cite{DoeDew12b}), the Born rule is captured by the topos approach to quantum theory. Concretely, if ``$\Ain\de$'' is a proposition, represented by $\ps\deo(\hP)\in\Subcl\SigN$, and $\mu_\rho:\Subcl\SigN\ra\cA(\VN,[0,1])$ is the probability measure corresponding to a state $\rho\in\mc S(\cN)$, then the minimum of the antitone function $\mu_{\rho}(\ps\deo(\hP)):\VN\ra [0,1]$ is the expectation value of finding the proposition to be true upon measurement of the physical quantity $A$, that is, the probability of finding the measurement outcome to lie in $\de$,
\begin{equation}
			\on{Prob}(\text{``}\Ain\de\text{''};\rho) = \min_{V\in\VN} (\mu_\rho(\ps\deo(\hP)))_V \stackrel{\eq{Eq_StatePropPairing}}{=} \min_{V\in\VN} (m_\rho(\ps\deo(\hP)))_V.
\end{equation}
If $V_{\hA}$ is a context that contains the self-adjoint operator $\hA$ representing the physical quantity $A$, e.g. $V_{\hA}=\{\hA,\hat 1\}''$, then the minimum is attained at $V_{\hA}$, so $\on{Prob}(\text{``}\Ain\de\text{''};\rho)=m_\rho(\ps\deo(\hP))_{V_{\hA}}$ \cite{Doe09a}. Prop. \ref{Prop_CompatibilityToposForm} implies in particular that the minima of the two antitone functions
\begin{equation}
			m_{\rho_t}(\ps\deo(\hP_0)),\;m_{\rho_0}(\ps\deo(\hP_t)):\VN\ra [0,1]
\end{equation}
are equal. This means that the expectation values, which are the physically relevant quantities, are equal for the Schr\"odinger picture, which is represented by $m_{\rho_t}(\ps\deo(\hP_0))$, and the Heisenberg picture, which is represented by $m_{\rho_0}(\ps\deo(\hP_t))$.

\paragraph{Covariance} The compatibility between the Schr\"odinger and the Heisenberg picture can alternatively be formulated in terms of a covariance relation. In standard quantum theory, the physical predictions do not change if we replace the state $\rho_0$ by $\rho_t=\hU_t.\rho_0=\rho_0\circ\phi_{\hU_{-t}}$ (cf. Def. \ref{Def_StatePresheafAtTimet}) and \emph{at the same time} replace each physical quantity $\hA_0$ by $\hU_t\hA_0\hU_{-t}$, hence standard quantum theory is covariant under the action of the unitary group (of the von Neumann algebra of observables). For example, if $\rho_0$ is a normal state with density matrix $\trho_0$,
\begin{align}
			\rho_0(\hA_0) &= \tr(\trho_0\hA_0)\\
			&= \tr(\hU_t\trho_0\hU_{-t}\hU_t\hA_0\hU_{-t})\\
			&= \rho_t(\hU_t\hA\hU_{-t}).
\end{align}
We will focus on projections here, so $\hA_0=\hP_0$. Then, by the usual convention (see footnote \ref{FN_TimeEvolutionOfProjs}), $\hU_t\hP_0\hU_{-t}=\hP_{-t}$ (and not $\hP_t$), so the covariance relation reads
\begin{equation}
			\rho_0(\hP_0)=\rho_t(\hP_{-t}).
\end{equation}
If $V\in\VN$ is a context that contains $\hP_0$, then $\hU_t V\hU_{-t}$ contains $\hP_{-t}$, so explicitly noting contexts as in equation \eq{Eq_CompatibilityStandardForm}, we have
\begin{equation}			\label{Eq_CovarianceStandardQT}
			\rho_0|_V(\hP_0)=\rho_t|_{\hU_t V\hU_{-t}}(\hP_{-t}).
\end{equation}

\begin{proposition}			\label{Prop_CovarianceToposForm}
Let $\rho_0\in\mc S(\cN)$ be a state, and let $m_{\rho_0}$ be the corresponding global section of $\CP$. Let $\ps S_0\in\Subcl\SigN$ be a clopen subobject, representing a proposition about the quantum system, e.g. $\ps S_0=\ps\deo(\hP_0)$, where $\hP_0$ represents the proposition ``$\Ain\de;0$''. Let $(\hU_t)_{t\in\bbR}$ be a strongly continuous one-parameter group of unitaries in $\cN$, let $m_{\rho_t}$ be the global section of $\CP$ corresponding to $\rho_t=\hU_t.\rho_0=\rho_0\circ\phi_{\hU_{-t}}$, and let $\ps S_{-t}$ be the clopen subobject
\begin{equation}			\label{Eq_S_-t}
			\ps S_{-t}:=\ptphi_{\hU_{-t}}(\ps S_0).
\end{equation}
Then, for all $V\in\VN$ and all $t\in\bbR$,
\begin{equation}			\label{Eq_CovarianceToposForm}
			(m_{\rho_0}(\ps S_0))_V = (m_{\rho_t}(\ps S_{-t}))_{\hU_t V\hU_{-t}}.
\end{equation}
\end{proposition}

\begin{proof}
We have, for all $V\in\VN$ and all $t\in\bbR$,
\begin{align}
			(m_{\rho_t}(\ps S_{-t}))_{\hU_t V\hU_{-t}} &\stackrel{\text{Def. \ref{Def_StatePropPairing}}}{=} m_{\rho_t;\hU_t V\hU_{-t}}(\ps S_{-t;\hU_t V\hU_{-t}})\\
			&\stackrel{\rho_t=\rho_0\circ\phi_{\hU_{-t}}}{=} m_{\rho_0\circ\phi_{\hU_{-t}};\hU_t V\hU_{-t}}(\ps S_{-t;\hU_t V\hU_{-t}})\\
			&\stackrel{\eq{Eq_ComponentsOfm_rho}}{=} (\rho_0\circ\phi_{\hU_{-t}})|_{\hU_t V\hU_{-t}}\circ\alpha_{\hU_t V\hU_{-t}}^{-1}(\ps S_{-t;\hU_t V\hU_{-t}})\\
			&= (\rho_0\circ\phi_{\hU_{-t}})|_{\hU_t V\hU_{-t}}(\hP_{\ps S_{-t;\hU_t V\hU_{-t}}})\\
			&= \rho_0(\hU_{-t}\hP_{\ps S_{-t;\hU_t V\hU_{-t}}}\hU_t)\\
			&\stackrel{\text{Cor. \ref{Cor_UnitaryActionOnProjs}}}{=} \rho_0(\hP_{\ptphi_{\hU_t}(\ps S_{-t})_V})\\
			&= \rho_0\circ\alpha_V^{-1}(\ptphi_{\hU_t}(\ps S_{-t})_V)\\
			&\stackrel{\eq{Eq_S_-t}}{=} \rho_0\circ\alpha_V^{-1}(\ps S_0;V)\\
			&\stackrel{\eq{Eq_ComponentsOfm_rho}}{=} m_{\rho_0}(\ps S_0)_V.
\end{align}
\end{proof}
Equation \eq{Eq_CovarianceToposForm} is the direct analogue of the covariance relation \eq{Eq_CovarianceStandardQT} of standard quantum theory.

\vspace{0.7cm}

\textbf{Acknowledgements.} Discussions with Chris Isham, Boris Zilber and Yuri Manin are gratefully acknowledged. Tom Woodhouse worked out a number of aspects in his M.Sc. thesis under my supervision. Rui Soares Barbosa pointed out some important issues, and Nadish de Silva, Dan Marsden and Carmen Constantin gave valuable feedback, for which I thank them. I also thank Bertfried Fauser, who carefully read a draft and made a number of very helpful remarks. I am grateful to John Harding, Izumi Ojima, John Maitland Wright and Dirk Pattinson for their interest in this work.

\end{document}